\documentclass[a4paper,12pt]{amsart}
\usepackage{graphicx}
\usepackage{mathrsfs}
\usepackage{amssymb}

\providecommand{\norm}[1]{\left\lVert #1\right\rVert}
\providecommand{\abs}[1]{\left\lvert #1\right\rvert}
\providecommand{\R}{\mathbb{R}}

\providecommand{\ee}{e}

\DeclareMathOperator{\tr}{tr}
\DeclareMathOperator{\Skew}{Skew}
\DeclareMathOperator{\Sym}{Sym}

\DeclareMathOperator{\Null}{Null}

\let\div\diver

\newtheorem{thm}{Theorem}[section]
\newtheorem{cor}[thm]{Corollary}
\newtheorem{lem}[thm]{Lemma}

\theoremstyle{remark}
\newtheorem{rem}[thm]{Remark}

\theoremstyle{definition}
\newtheorem{dfn}[thm]{Definition}

\theoremstyle{remark}

\begin{document}

\title[Stokes flow]{Stokes flow with kinematic and dynamic boundary conditions}

\author[J.~Fabricius]{John Fabricius}
\address{%
Department of Engineering Sciences and Mathematics\\
Lule\aa\ University of Technology\\
SE-971 87 Lule\aa , Sweden}
\email{john.fabricius@ltu.se}

\keywords{Stokes equation, stress condition, traction condition, Neumann condition, pressure operator}

\subjclass[2010]{76D03,76D07}


\begin{abstract}
We review the first and second boundary value problems for the Stokes system posed in a bounded Lipschitz domain in $\R^n.$ Particular attention is given to the mixed boundary condition: a Dirichlet condition is imposed for the velocity on one part of the boundary while a Neumann condition for the stress tensor is imposed on the remaining part. Some minor modifications to the standard theory are therefore required. The most noteworthy result is that both pressure and velocity are unique.
\end{abstract}

\maketitle

\section{Introduction}

Let $\Omega$ be a bounded domain in $\R^n$ ($n\geq 2$), whose boundary $\partial \Omega,$ with outward unit normal $\hat{n},$ is divided into two disjoint parts: $\Gamma_D$ and $\Gamma_N.$ The present work concerns the solvability of the boundary value problem (b.v.p.)
\begin{equation}\label{Stokes:bvp}
\left\{
\begin{aligned}
 \div \sigma& =f && \text{in }\Omega && \text{(1a)}\\
 \div{u}& =0 && \text{in }\Omega && \text{(1b)}\\
  \sigma\hat{n}& =g && \text{on }\Gamma_N && \text{(1c)}\\
u& =h && \text{on }\Gamma_D && \text{(1d)}.
\end{aligned}\right.
\end{equation}
The symbol $\sigma=(\sigma_{ij})_{ij}$ denotes the Cauchy stress tensor which is defined by the constitutive law
\begin{equation}\label{S:rel}
\sigma_{ij}=-p\,\delta_{ij}+\mu\left(\frac{\partial u_i}{\partial x_j}+\frac{\partial u_j}{\partial x_i}\right)
\quad (1\leq i,j\leq n),
\end{equation}
where the dynamic viscosity $\mu>0$ is a given constant. Under this assumption, (\ref{Stokes:bvp}a--b) becomes the Stokes system which describes the steady flow of an incompressible Newtonian fluid when inertial forces can be neglected. Moreover, we assume that $f$ (body forces), $g$ (surface forces) and $h$ (surface velocity) are given vector functions and we seek the velocity $u$ (vector function) and the pressure $p$ (scalar function) of the fluid.

The b.v.p. \eqref{Stokes:bvp} when $\Gamma_D=\partial \Omega$ and $\Gamma_N=\varnothing$ is classical and has been studied by many authors. The present work concerns the mixed boundary condition (\ref{Stokes:bvp}c-d). The Dirichlet condition (\ref{Stokes:bvp}d) is a \emph{kinematic} condition that prescribes the velocities of fluid particles at each point of the surface $\Gamma_D.$ It is natural to impose this condition if $\Gamma_D$ is the contact surface between a fluid and a body moving with a prescribed velocity, provided that all components of the velocity vector are continuous across the interface. This so called no-slip condition is valid at solid boundaries but also at the interface between two similar fluids \cite[pp. 12--14]{Pan96}. We do not assume that $\Gamma_D$ is impermeable, so $h$ may have a normal component --- fluid particles may enter or leave the domain through any part of the boundary as long as (\ref{Stokes:bvp}b) is fulfilled. The Neumann condition (\ref{Stokes:bvp}c), on the other hand, is a \emph{dynamic} condition that prescribes the stress vector (traction) at each point of $\Gamma_N.$ In physical terms, it means that the fluid is subjected to an externally applied surface force as a result of the mechanical contact with a surrounding body. The dynamic condition is applicable whenever the stress vector is continuous across a surface. In particular, this is true when $\Gamma_N$ is the interface between two fluids (not necessarily of the same constitutive type), see \cite[Art.~327]{Lamb32} or \cite[p.~240]{Ser59}. For a more general discussion of fluid-fluid interfaces we refer to Chapter 5 of \cite{Pan96}.

The kinematic condition is treated in all classical texts devoted to the mathematical study of fluids \cite{Galdi11,Lad69,Lions69,Tem84}. But there are only a few works related to the dynamic condition in the mathematical literature. In the related field of elasticity theory, equal attention is given to both conditions (see e.g. \cite{Duv76,Mar83}). A strong argument for considering the mixed condition for fluids comes from the solvability of \eqref{Stokes:bvp} alone. Under fairly general assumptions on $f, g$ and $h$ it can be shown that the b.v.p. \eqref{Stokes:bvp} has a weak solution:
\begin{enumerate}
\item If $\int_{\Gamma_N}dS=0$ and $\int_{\partial \Omega} h\cdot \hat{n}\,dS=0,$ there exists a solution $(u,p)$ of \eqref{Stokes:bvp}. The velocity is unique but the pressure is only unique up to a constant.
\item If $\int_{\Gamma_D}dS=0$ and $f$ is ``compatible'' with $g,$ there exists a solution $(u,p)$ of \eqref{Stokes:bvp}. The pressure is unique but the velocity is only unique up to the motion of a rigid body.
\item If $\int_{\Gamma_D}dS>0$ and $\int_{\Gamma_N}dS>0$, there exists a solution $(u,p)$ of \eqref{Stokes:bvp} for ``arbitrary'' data $f,$ $g$ and $h.$ Both the velocity and the pressure are unique.
\end{enumerate}
This result has practical implications. In many applications involving viscous flows, e.g. in lubrication theory, the main problem is to calculate the net force exerted by the fluid on a solid boundary. For such problems the kinematic condition is obviously deficient, inasmuch as the dynamic state of the fluid cannot be completely determined. A remedy is to impose a dynamic condition on some part of the boundary. On a more general note, if one believes that the Stokes system (\ref{Stokes:bvp}a--b) is a good model for the creeping motion of a viscous fluid in the real world, where both velocity and stress are observable quantities, the mixed condition seems like the most realistic alternative.

A precise formulation of the solvability of \eqref{Stokes:bvp} when $\Omega$ is a bounded domain in $\R^n$ and $\partial \Omega$ is of class $C^{0,1},$  is the principal result of the present work. For the Dirichlet problem, one deduces the existence of a pressure function from De Rham's theorem \cite{Tar78,Tar06}. The mixed boundary condition, however, requires a \emph{stronger} formulation of this result. Nevertheless, the dynamic condition has been studied by some authors. The first occurrences that we have found are in the works of Glowinski \cite{Glo84,Glo03} and Pironneau \cite{Pir89}. An existence and uniqueness result for a generalized version of \eqref{Stokes:bvp}, when $\partial \Omega$ is smooth, is proved in Chapter~IV of \cite{Glo03}. For the existence of a pressure function, the author introduces a so called ``Stokes operator'' which maps $L^2(\Omega)$ into $L^2(\Omega).$ Our approach differs, in that we define the pressure operator in terms of the Bogovski\u{\i} operator, which is a well established concept in the mainstream literature. This gives a more precise estimate for the $L^2$-norm of the pressure, but Glowinski's approach could be more suitable for numerical computations. The solvability of \eqref{Stokes:bvp} when $n=3$ and $\Gamma_N=\partial \Omega$ is of class $C^{1,1}$ is studied by Boyer and Fabrie in Chapter~IV of \cite{Boy13} (without references to previous work). The authors observe that the pressure is unique, but do not consider the mixed boundary condition. Maz'ya and Rossmann \cite{Maz07,Maz09} have studied the stationary Stokes system in a three-dimensional domain of polyhedral type with one of four different boundary conditions, including (\ref{Stokes:bvp}c--d), imposed on each face. They are mainly concerned with regularity assertions of weak solutions, but existence and uniqueness of the mixed boundary value problem for the Stokes system is proved in Theorem~5.1 of \cite{Maz07}. The present paper extends the aforementioned results to bounded Lipschitz domains in $\R^n.$ Moreover we clarify the notion of a pressure operator.

The paper is structured as follows. In Section~\ref{prel:sec} we define the fluid domain as well as the relevant function spaces in order to formulate a notion of weak solution for the mixed b.v.p. \eqref{Stokes:bvp}. The main result, Theorem~\ref{main:thm}, is stated in Section~\ref{mainres:sec}, but the proof is postponed until Section~\ref{proof:sec}. Section~\ref{ineq:sec} is devoted to inequalities. We state three versions of Korn's inequality, which is closely related to Ne\v{c}as' inequality. The third Korn inequality (Theorem~\ref{Korn3:thm}) is a key result. So is the pressure operator described in Section~\ref{vector_analysis:sec}, which follows from a theorem due to Bogovski\u{\i}. We conclude this paper with a simple application of Theorem~\ref{main:thm} to ``pressure-driven'' flow in a channel bounded by two parallell plates. We compute the solution of this problem using a finite element program and discuss its relation to the classical Poiseuille solution.

\section{Preliminaries and notation}\label{prel:sec}

\subsection{Euclidian structure}

Let $\R^{m\times n}$ denote the set of real $m\times n$ matrices $X=(x_{ij})_{ij}$ equipped with the Euclidian scalar product
\begin{equation*}
X:Y=\tr (X^TY)=\sum_{j=1}^n \sum_{i=1}^m x_{ij}y_{ij},\quad \abs{X}=(X:X)^{1/2}
\end{equation*}
where $X^T=(x_{ji})_{ij}$ in $\R^{n\times m}$ denotes the transpose of $X.$  For $X\in \R^{n\times n},$ the symmetric and skew-symmetric parts of $X$ are defined respectively as
\begin{equation*}
\ee(X)=\frac{1}{2}\left( X+X^T\right),\quad \omega(X)=\frac{1}{2}\left( X-X^T\right).
\end{equation*}
The subspaces of symmetric and skew-symmetric matrices are defined as
\begin{align*}
\Sym(n)& =\Null \omega=\left\{X\in \R^{n\times n}: \omega(X)=0\right\}\\
\Skew(n)& =\Null \ee=\left\{X\in \R^{n\times n}: \ee(X)=0\right\}.
\end{align*}
The identity element in $\R^{n\times n}$ is denoted as $I=(\delta_{ij})_{ij}.$ 
We identify $\R^n$ with $\R^{n\times 1}$ (column vectors) and denote the scalar product in $\R^n$ as
\begin{equation*}
x\cdot y=x^Ty=\sum_{i=1}^n x_iy_i,\quad \abs{x}=(x\cdot x)^{1/2}.
\end{equation*}
The standard basis vectors in $\R^n$ are denoted as $e_1,\dotsc, e_n.$

\subsection{Fluid domain}
The fluid domain $\Omega$ is assumed to be an open bounded connected subset of $\R^n$ with a Lipschitz boundary $\partial \Omega.$ By the latter condition we mean that there exists a finite collection of open sets $\{U_i\},$ to each member of which there corresponds a rotation $R_i$ and a function $f_i\in C^{0,1}(\R^{n-1})$ such that
\begin{equation*}
\partial \Omega\subset \bigcup_i U_i\quad
\text{and}\quad
U_i\cap \Omega=U_i\cap R_iH_{f_i},
\end{equation*}
where
\begin{equation*}
H_{f_i}=\left\{y\in \R^n: y_n>f_i(y_1,\dotsc,y_{n-1})\right\}.
\end{equation*}
Consequently, the outward unit normal $\hat{n}$ is defined almost everywhere on $\partial \Omega$ and the divergence theorem holds, i.e.
\begin{equation}\label{div:thm}
\int_{\partial \Omega} v\cdot \hat{n}\,dS=\int_{\Omega} \div v\,dx\quad \forall v\in C^1(\R^n;\R^n),
\end{equation}
where $dS$ denotes surface measure on $\partial \Omega.$  $\Gamma_D$ and $\Gamma_N$ are assumed to be $dS$-measurable disjoint sets such that $\partial \Omega=\Gamma_D\cup \Gamma_N.$ Further regularity assumptions on $\Gamma_D$ and $\Gamma_N$ are imposed below.

\subsection{Function spaces}

The dual of a Banach space  $V$ is denoted as $V'$ and the symbol $\langle L,v\rangle$  stands for the evaluation of $L\in V'$ at $v\in V.$ If $Y$ is subspace of $V,$ the quotient space $V/Y$ is defined in the usual way and $\{v\}$ denotes the equivalence class in $V/Y$ to which $v\in V$ belongs.

We shall work mainly with the following spaces
\begin{align*}
W(\Omega)& =\bigl\{v\in H^1(\Omega;\R^n): v=0\text{ on }\Gamma_D\bigr\}\\
V(\Omega)& =\bigl\{v\in W(\Omega): \div v=0\text{ in }\Omega\bigr\}.
\end{align*}
Clearly $V(\Omega)\subset W(\Omega)$ are closed subspaces of $H^1(\Omega;\R^n).$ Moreover, $W/V(\Omega)$ is defined as the quotient of $W(\Omega)$ by $V(\Omega),$ equipped with the norm
\[
\norm{\{u\}}_{W/V(\Omega)}=\inf_{v\in V(\Omega)} \norm{u-v}_{H^1(\Omega)}.
\]

By definition, $H^{1/2}(\Gamma_D;\R^n)$ consists of all elements in $H^{1/2}(\partial\Omega;\R^n)$ restricted to $\Gamma_D.$ On $H^{1/2}(\Gamma_D;\R^n)$ we choose the norm
\begin{equation*}
\norm{u}_{H^{1/2}(\Gamma_D)}=\norm{\{u\}}_{H^1/W(\Omega)}=\inf_{v\in W(\Omega)} \norm{u-v}_{H^1(\Omega)}.
\end{equation*}
Moreover we define the space $H_0^{1/2}(\Gamma_N;\R^n)$ as the image of $W(\Omega)$ under the trace operator: $H^1(\Omega;\R^n)\to H^{1/2}(\partial \Omega;\R^n).$ $H_0^{1/2}(\Gamma_N;\R^n)$ is equipped with the norm of $W/H_0^1(\Omega;\R^n).$ The dual of $H_0^{1/2}(\Gamma_N;\R^n)$ is denoted as $H^{-1/2}(\Gamma_N;\R^n).$ The following  characterization will prove to be useful: For any $L\in H^{-1/2}(\Gamma_N;\R^n)$ there exists $G\in L^2(\Omega;\R^{n\times n})$ with $\div G\in L^2(\Omega;\R^n)$ such that
\begin{equation}\label{GN:dual}
\langle L,v\rangle_{H^{-1/2}(\Gamma_N),H_0^{1/2}(\Gamma_N)}=\int_\Omega G:\nabla v+(\div G)\cdot v\,dx \quad \forall v\in W(\Omega). 
\end{equation}
By analogy with the divergence theorem we say that $L=G\hat{n}$ on $\Gamma_N,$ whenever \eqref{GN:dual} holds.
Moreover
\begin{equation}
\norm{L}_{H^{-1/2}(\Gamma_N)}=\inf \left(\int_\Omega \abs{G}^2+\abs{\div G}^2\,dx\right)^{1/2}
\end{equation}
where the infimum is taken over all $G$ satisfying \eqref{GN:dual}.

We denote the space of rigid body velocities as
\begin{equation*}
R(\Omega)=\bigl\{v\in H^1(\Omega;\R^n):\ee(\nabla v)=0 \text{ in } \Omega\bigr\}.
\end{equation*}
Since $\Omega$ is connected, each element $v$ in $R(\Omega)$ can be represented as $v(x)=Ax+b,$ for some $A\in \Skew(n)$ and $b\in \R^n.$

\subsection{Weak formulation}

In view of the notation introduced above, we shall subsequently write the Newton--Stokes law \eqref{S:rel} as $\sigma=-p\,I+2\mu\,\ee(\nabla u).$

\begin{dfn}
Assume $f\in L^2(\Omega; \R^n),$ $g\in L^2(\Gamma_N;\R^n)$ and $h\in H^{1/2}(\Gamma_D;\R^n).$  We say that $u\in H^1(\Omega;\R^n)$ and $p\in L^2(\Omega)$ are a weak solution of \eqref{Stokes:bvp} if $\div u=0$ in $\Omega,$ $u=h$ on $\Gamma_D$ and
\begin{equation}\label{Stokes:weak}
\int_{\Gamma_N} g\cdot v\,dS=\int_\Omega \bigl(-p\, I+2\mu\,\ee(\nabla u)\bigr):\nabla v\,dx+\int_\Omega f\cdot v\,dx
\end{equation}
for all $v$ in $W(\Omega).$ 
\end{dfn}

\begin{proof}[Derivation of weak formulation]
By formally applying the divergence theorem we obtain
\begin{equation}
\begin{split}
\int_{\partial \Omega} \sigma\hat{n}\cdot v\,dS
& =\int_\Omega \div (\sigma^Tv)\,dx\\
& =\int_\Omega \sigma:\nabla v+(\div \sigma)\cdot v\,dx.
\end{split}
\end{equation}
Taking into account (\ref{Stokes:bvp}a,c) and $v=0$ on $\Gamma_D$ gives \eqref{Stokes:weak}.
\end{proof}

In the weak formulation \eqref{Stokes:weak}, the sets $\Gamma_D$ and $\Gamma_N$ are defined only up to a set of measure zero. It follows that the extreme case $\int_{\Gamma_N}dS=0$ is equivalent to $\Gamma_D=\partial \Omega$ and $\Gamma_N=\varnothing,$ as $W(\Omega)$ then coincides with $H_0^1(\Omega;\R^n).$ Similarly $\int_{\Gamma_D}dS=0$ is equivalent to $\Gamma_D=\varnothing$ and $\Gamma_N=\partial\Omega,$ as $W(\Omega)$ then coincides with $H^1(\Omega;\R^n).$

\begin{rem}
Note that it is possible to take $g$ in the larger space $H^{-1/2}(\Gamma_N;\R^n)$ provided that one replaces the surface integral of \eqref{Stokes:weak} with
\begin{equation*}
\langle g,v\rangle_{H^{-1/2}(\Gamma_N),H_0^{1/2}(\Gamma_N)}.
\end{equation*}
Similarly, one can take $f$ in $W'(\Omega)$ by replacing the second volume integral with the more general expression
\begin{equation*}
\langle f,v\rangle_{W'(\Omega),W(\Omega)}.
\end{equation*}
\end{rem}

\section{Main result}\label{mainres:sec}

We state here the main result concerning the solvability of the b.v.p. \eqref{Stokes:bvp}. As mentioned in the introduction, we are mainly concerned with part (iii).

\begin{thm}\label{main:thm}
Let $\Gamma_D$ be a closed subset of $\partial \Omega$ and define $\Gamma_N$ as the complement of $\Gamma_D$ in $\partial \Omega.$
Let $f\in L^2(\Omega; \R^n),$ $g\in L^2(\Gamma_N;\R^n)$ and $h\in H^{1/2}(\Gamma_D;\R^n)$ be given functions. 
\begin{enumerate}
\item Assume $\Gamma_D=\partial \Omega$ and
\begin{equation}\label{h:comp}
\int_{\partial \Omega} h\cdot \hat{n}\,dS=0.
\end{equation}
Then there exists a unique weak solution
\begin{equation*}
u\in H^1(\Omega; \R^n),\quad \{p\}\in L^2(\Omega)/\R
\end{equation*}
of \eqref{Stokes:bvp}
such that
\begin{equation*}
\mu\norm{u}_{H^1(\Omega)}+\norm{\{p\}}_{L^2(\Omega)/\R}\leq {C}\bigl(\norm{f}_{L^2(\Omega)}+\norm{h}_{H^{1/2}(\partial\Omega)}\bigr),
\end{equation*}
where the constant $C$ depends only on $\Omega.$
\item Assume $\Gamma_D=\varnothing$ and
\begin{equation}\label{g:comp}
\int_{\partial \Omega} g\cdot v\,dS=\int_\Omega f\cdot v\,dx\quad \forall v\in R(\Omega).
\end{equation}
Then there exists a unique weak solution
\begin{equation*}
\{u\}\in H^1(\Omega; \R^n)/R(\Omega),\quad p\in L^2(\Omega)
\end{equation*}
of \eqref{Stokes:bvp} such that
\begin{equation*}
\mu\norm{\{u\}}_{H^1/R(\Omega)}+\norm{p}_{L^2(\Omega)}\leq {C}\bigl(\norm{f}_{L^2(\Omega)}+\norm{g}_{H^{-1/2}(\partial\Omega)}\bigr),
\end{equation*}
where the constant $C$ depends only on $\Omega.$
\item Assume $\int_{\Gamma_D}dS>0$ and $\int_{\Gamma_N}dS>0.$
Then there exists a unique weak solution
\begin{equation*}
u\in H^1(\Omega; \R^n),\quad p\in L^2(\Omega)
\end{equation*}
of \eqref{Stokes:bvp} such that
\begin{equation*}
\mu\norm{u}_{H^1(\Omega)}+\norm{p}_{L^2(\Omega)}\leq {C}\bigl(\norm{f}_{L^2(\Omega)}+\norm{g}_{H^{-1/2}(\Gamma_N)}+\norm{h}_{H^{1/2}(\Gamma_D)}\bigr),
\end{equation*}
where the constant $C$ depends only on $\Omega$ and $\Gamma_D.$
\end{enumerate}
\end{thm}

\begin{rem}
The regularity hypothesis in part (iii) of Theorem~\ref{main:thm} that $\Gamma_D$ be closed (with $\Gamma_N$ relatively open) deserves an explanation. Clearly it is not a necessary condition. The main restriction comes from Lemma~\ref{psi0:lem}, which assumes that $\Gamma_N$ have an interior point. Nevertheless $\Gamma_D$ can be quite irregular, e.g. nowhere dense in $\partial \Omega.$
\end{rem}

\begin{rem}
Note that part (i) of Theorem~\ref{main:thm} holds only if $h$ is compatible with $\div u$ through \eqref{h:comp}. Similarly, part (ii) holds only if $g$ is compatible with $\div \sigma$ through \eqref{g:comp}.  In contrast, part (iii) requires no such compatibility condition. 
\end{rem}

\section{Inequalities}\label{ineq:sec}

A fundamental result in analysis is Ne\v{c}as' inequality, see \cite{Nec66} or \cite[Lemma~7.1]{Nec12}. Consider the space
\begin{equation*}
X(\Omega)=\left\{T\in H^{-1}(\Omega):\nabla T\in H^{-1}(\Omega;\R^n)\right\}.
\end{equation*}
Clearly $L^2(\Omega)\subset X(\Omega),$ but the stronger assertion $X(\Omega)=L^2(\Omega)$ requires some regularity of  $\partial \Omega.$ In particular it is true if $\Omega$ has a compact Lipschitz boundary, which is the present case.

\begin{thm}[Ne\v{c}as' inequality]
There exists a constant $C$ depending only on $\Omega$ such that
\begin{equation*}\label{Nec:ineq}
\norm{p}_{L^2(\Omega)}\leq C\bigl(\norm{p}_{H^{-1}(\Omega)}+\norm{\nabla p}_{H^{-1}(\Omega)}\bigr)\quad \forall p\in L^2(\Omega)
\end{equation*}
and $X(\Omega)=L^2(\Omega).$
\end{thm}


Several other important results and inequalities stated below are easily deduced from Ne\v{c}as' inequality. Among them the following (see Tartar \cite{Tar78,Tar06} for details).

\begin{cor}\label{range_grad:cor}
The range of $\nabla\colon L^2(\Omega)\to H^{-1}(\Omega;\R^n)$ is equal to
\begin{equation*}
(\Null\div)^\perp=\left\{T\in H^{-1}(\Omega;\R^n):\langle T,v\rangle=0\quad \forall v\in V_0\right\}
\end{equation*}
\end{cor}

Since $-\div$ is the dual operator of $\nabla$  we have the following immediate corollary. 

\begin{cor}\label{range_div:cor}
The range of $\div\colon H_0^{1}(\Omega;\R^n)\to L^2(\Omega)$ is equal to
\begin{equation*}
(\Null \nabla)^\perp=\left\{p\in L^2(\Omega):\int_\Omega p\,dx=0\right\}.
\end{equation*}
\end{cor}

Ne\v{c}as' inequality is intimately related to Korn's inequality. We state below three versions of the Korn inequality. The third version is the most relevant one for the mixed b.v.p. \eqref{Stokes:bvp}.

\begin{thm}[First Korn inequality]\label{Korn1:thm}
There exists a constant $K$ depending only on $\Omega$ such that
\begin{equation*}
\norm{\nabla v}_{L^2(\Omega)}\leq K\left(\norm{v}_{L^2(\Omega)}^2+\norm{\ee(\nabla v)}_{L^2(\Omega)}^2\right)^{1/2}\quad \forall v\in H^1(\Omega;\R^n).
\end{equation*} 
\end{thm}

For a proof based on Ne\v{c}as' inequality, see Theorem~3.1 in Chapter~III of Duvaut and Lions \cite{Duv76}. The authors assume that $\partial \Omega$ is of class $C^1,$ but their proof is valid when $\partial \Omega$ is of class $C^{0,1},$ as observed by Ciarlet \cite{Cia10}.


\begin{thm}[Second Korn inequality]\label{Korn2:thm}
There exists a constant $K$ depending only on $\Omega$ such that
\begin{equation*}
\norm{\{v\}}_{H^1(\Omega)/R}
\leq K\norm{\ee(\nabla v)}_{L^2(\Omega)}\quad \forall v\in H^1(\Omega;\R^n).
\end{equation*} 
\end{thm}

For a proof, see Theorem~3.4 of \cite{Duv76} or Theorem~2.3 (a) of \cite{Cia10}.


To prove the third version of Korn's inequality we need an auxiliary result from classical mechanics.

\begin{lem}\label{hyp:lem}
Let $\Gamma$ be a hypersurface in $\R^n$ defined by
\begin{equation*}
\Gamma=\bigl\{x\in \R^n: x_n=f(x_1,\dotsc,x_{n-1}),\quad f\in C^{0,1}(\R^{n-1})\bigr\}.
\end{equation*}
Let $v$ be a rigid body velocity that vanishes on a set $E\subset \Gamma$ of positive surface measure. Then $v$ is identically zero.
\end{lem}

\begin{proof}
Assume $v(x)=Ax+b,$ where $A\in \Skew(n)$ and $b\in \R^n.$ Since $v$ vanishes on $E$ we have
\begin{equation}\label{Axb:eq}
\int_E \phi\,Ax\,dS+\int_E \phi\,b\,dS=0\quad \forall \phi\in L^2(E,dS),
\end{equation}
where $dS=\sqrt{1+\abs{\nabla f(x')}^2}\,dx',$ $x'=(x_1,\cdots,x_{n-1}).$ Let $\mathcal{P}$ be the subspace of $L^2(E,dS)$ which is spanned by the polynomials $p_0,\dotsc,p_{n-1}$ defined by
\begin{equation*}
p_i(x)=
\begin{cases}
1 & (i=0)\\
x_i & (1\leq i\leq n-1).
\end{cases}
\end{equation*}
Since $E$ has positive measure there exists a basis $q_0,\dotsc q_{n-1}$ of $\mathcal{P}$ such that
\begin{equation}\label{qipj}
\int_E q_ip_j\,dS=\delta_{ij}\quad (0\leq i,j\leq n-1).
\end{equation}
Taking $\phi=q_0$ in \eqref{Axb:eq} we deduce $b=0.$ Let $y_1,\dotsc,y_{n-1}$ be the vectors in $\R^n$ defined by
\begin{equation*}
y_i=\int_{E} q_i\,x\,dS
\quad (i=1,\dotsc, n-1).
\end{equation*}
By \eqref{qipj}, $y_1,\dotsc,y_{n-1}$ are linearly independent. Taking $\phi=q_i$ in \eqref{Axb:eq} we deduce
\[
Ay_i=0\quad (i=1,\dotsc, n-1),
\]
so $\dim \Null A\geq n-1.$ Since $A$ is skew-symmetric we conclude that $A=0.$
\end{proof}

\begin{rem}
If $\Gamma_D$ has an interior point, the proof of Lemma~\ref{hyp:lem} can be slightly simplified.
\end{rem}

%
%

\begin{thm}[Third Korn inequality]\label{Korn3:thm}
Assume $\int_{\Gamma_D}dS>0.$ Then there exists a constant $K$ depending only on $\Omega$ and $\Gamma_D$ such that
\begin{equation}\label{Korn3:ineq}
\norm{v}_{H^1(\Omega)}\leq K\norm{\ee(\nabla v)}_{L^2(\Omega)}\quad \forall v\in W(\Omega).
\end{equation} 
\end{thm}

\begin{proof}
The proof is by contradiction. Suppose \eqref{Korn3:ineq} were false. Then for each positive integer $m$ there exists $v_m$ in $W(\Omega)$ such that $\norm{v_m}_{H^1}=1$ and
\begin{equation}\label{Tfalse:ineq}
\norm{\ee(\nabla v_m)}_{L^2}\leq \frac{1}{m}.
\end{equation}
By the Rellich-Kondrachov theorem there exists a subsequence, also denoted as $v_{m},$ that converges strongly in $L^2(\Omega;\R^n).$ Applying the first Korn inequality (Theorem~\ref{Korn1:thm}) to $v_{m+k}-v_{m},$ $k>0,$ we deduce from \eqref{Tfalse:ineq} that $v_m$ is a Cauchy sequence in $H^1(\Omega;\R^n).$  Thus $v_m$ converges strongly to some $v$ in $H^1(\Omega; \R^n).$ Letting $m\to \infty$ in \eqref{Tfalse:ineq} gives $\ee(\nabla v)=0,$ so we have $v=r$ a.e. in $\Omega$ for some $r\in R(\Omega).$ Since $r$ vanishes on a subset of $\partial \Omega$ of positive surface measure, Lemma~\ref{hyp:lem} implies $r=0.$ This is a contradiction, because $\norm{v}_{H^1}=1.$\end{proof}

\begin{rem}
Theorem~\ref{Korn3:thm} is formulated as Theorem~3.3 in \cite{Duv76} for $\partial \Omega$ of class $C^{1}.$ The proof presented here is merely a reproduction that proof, except for a minor detail: The implication
\begin{equation*}
v\in R(\Omega)\cap \{v\in H^1(\Omega;\R^n): v=0\text{ on }\Gamma_D\}\implies v=0
\end{equation*}
provided $\int_{\Gamma_D}dS>0,$ is not proved in \cite{Duv76}. Lemma~\ref{hyp:lem} bridges this gap. 
\end{rem}

%

\section{Some results in vector analysis}\label{vector_analysis:sec}

The main aim of this section is to establish the existence of a pressure operator, a.k.a. de Rham operator, which is adapted to the b.v.p. \eqref{Stokes:bvp}.  First, we establish that the space $H_0^{1/2}(\Gamma_N;\R^n)$ is non-trivial under present assumptions.

\begin{lem}\label{psi0:lem}
Assume $\Gamma_D$ and $\Gamma_N$ as in Theorem~\ref{main:thm}(iii). Then there exist 
\begin{enumerate}
\item $\rho\in C^{0,1}(\R^n)$ such that
\begin{equation*}
\rho=0\quad\text{on }\Gamma_D\quad \text{and}\quad 
\int_{\partial \Omega} \rho\, dS>0,
\end{equation*}

\item $v\in C_c^\infty(\R^n;\R^n)$ and a constant $c>0$ depending only on $\partial\Omega$ such that
\begin{equation*} 
c\leq v\cdot \hat{n}\leq 1\quad \text{a.e. on }\partial \Omega,
\end{equation*}

\item $\hat{v}\in H^{1/2}(\partial \Omega;\R^n)$ such that
\begin{equation*} 
\hat{v}=0\quad \text{on }\Gamma_D\quad \text{and}\quad \int_{\partial \Omega} \hat{v}\cdot \hat{n}\,dS=1.
\end{equation*}
\end{enumerate}
\end{lem}

\begin{proof}
(i)\enspace Given $x\in \R^n,$ let $\rho(x)$ denote the distance from $x$ to the set $\Gamma_D,$ i.e.
\begin{equation*}
\rho(x)=\inf_{y\in \Gamma_D} \abs{x-y}\quad (x\in \R^n).
\end{equation*}
Clearly $\rho\in C^{0,1}(\R^n)$ with Lipschitz constant equal to one. Since $\Gamma_N$ is relatively open and non-empty, we deduce
\begin{equation*}
\int_{\partial \Omega} \rho\,dS>0.
\end{equation*}

\noindent (ii)\enspace Let $\{U_i\}$ be the open cover used in the definition of $\partial \Omega$ with corresponding rotations $\{R_i\}$ and Lipschitz functions  $\{f_i\}.$ Set $v_i=-R_ie_n$ and let $L_i$ denote the Lipschitz constant of $f_i.$ It is readily checked that
\begin{equation*}
\frac{1}{\sqrt{1+L_i^2}}
\leq v_i\cdot \hat{n}(x)\leq 1
\end{equation*}
at each point $x\in \partial \Omega\cap U_i$ where $\hat{n}(x)$ is defined. Let $\{\varphi_i\}$ be a smooth partition of unity of $\partial \Omega$ subordinate to $\{U_i\}$ and define
\begin{equation*}
v(x)=\sum_i \varphi_i(x)v_i,\quad c=\min_i \frac{1}{\sqrt{1+L_i^2}}.
\end{equation*}
It follows that $c\leq v(x)\cdot \hat{n}(x)\leq 1$ for a.e. $x\in \partial \Omega.$

\noindent (iii)\enspace Let $\rho$ and $v$ be as in (i) and (ii) respectively. Clearly there exists $\hat{v}$ of the form $\hat{v}(x)=\rho(x)v(x)$ such that $\int_{\partial \Omega} \hat{v}\cdot \hat{n}\,dS=1.$
\end{proof}

\begin{rem}
Part (ii) of Lemma~\ref{psi0:lem} is almost as Lemma~1.5.1.9 of \cite{Gri85}. The proof is also the same.
\end{rem}

\begin{rem}\label{psi0:rem}
If $\Gamma_D=\varnothing$ there are simpler choices of $\hat{v}$ than the above construction, e.g. one can choose $\hat{v}$ as $\hat{v}(x)=(n\abs{\Omega})^{-1}x.$ In order to choose $\hat{v}$ as $\rho\,\hat{n},$ which is perhaps the most obvious choice, $\partial \Omega$ must be of class $C^{1,1}.$
\end{rem}

Corollary~\ref{range_div:cor} is often attributed to Bogovski{\u\i} \cite{Bog79} (see also Chapter III of \cite{Galdi11}). We shall prove a similar statement which is more appropriate for the mixed boundary condition. Namely, that there exists a bounded linear operator
\[
B\colon L^2(\Omega)\to W/V(\Omega),
\]
called the Bogovski{\u\i} operator, which is the inverse of the divergence operator when its nullspace is collapsed to zero.

\begin{thm}[Bogovski{\u\i}]\label{range_div:thm}
Assume $\Gamma_D$ and $\Gamma_N$ as in Theorem~\ref{main:thm}(iii). Then for any $f\in L^2(\Omega)$ there exists a solution $u\in H^1(\Omega;\R^n)$ of the b.v.p.
\begin{equation*}
\left\{
\begin{aligned}
\div u& =f\quad \text{in }\Omega\\
u& =0\quad \text{on } \Gamma_D.
\end{aligned}\right.
\end{equation*}
Furthermore, $u$ is unique in $W(\Omega)$ modulo $V(\Omega)$ with
\begin{equation}\label{Bogineq}
\norm{\{u\}}_{W/V(\Omega)} \leq C\norm{f}_{L^2(\Omega)},
\end{equation}
where the constant $C$ depends only on $\Omega$ and $\Gamma_D.$
\end{thm}

\begin{proof}
Given $f\in L^2(\Omega),$ set $a=\int_\Omega f\,dx$ and choose $\hat{v}$ as in Lemma~\ref{psi0:lem}. Since
\begin{equation*}
\int_\Omega (f-a\div\hat{v})\,dx=0,
\end{equation*}
Corollary~\ref{range_div:cor} asserts there exists $w\in H_0^1(\Omega;\R^n)$ such that
\begin{equation*}
f-a\div\hat{v}=\div w.
\end{equation*}
Thus $u=w+a\hat{v}$ belongs to $W(\Omega)$ and satisfies
$f=\div u.$

Let $A\colon W/V(\Omega)\to L^2(\Omega)$ be the linear operator defined by
\begin{equation*}
A\{u\}=\div u.
\end{equation*}
$A$ is continuous, because
\begin{equation*}
\norm{A\{u\}}_{L^2}=\norm{\div u}_{L^2}\leq \sqrt{n}\norm{\nabla u}_{L^2}\quad \forall u\in\{u\} 
\end{equation*}
which implies $\norm{A}\leq \sqrt{n}.$  Since $A$ is one-to-one and onto, the open mapping theorem asserts that its inverse $B\colon L^2(\Omega)\to W/V(\Omega)$ is continuous. Choose  $u\in W(\Omega)$ so that $f=\div u,$ which is equivalent to $\{u\}=Bf.$ Then
\begin{equation*}
\norm{\{u\}}_{W/V}=\norm{Bf}_{W/V}\leq \norm{B}\norm{f}_{L^2}.
\end{equation*}
\end{proof}

Note, that the constant $C$ in Theorem~\ref{range_div:thm} is actually the norm of the operator $B$ which is defined as
\[
\norm{B}=\sup_{\substack{f\in L^2(\Omega)\\ \norm{f}=1}} \inf_{\substack{v\in W(\Omega)\\ \div v=f}}
\norm{v}_{H^1(\Omega)}.
\]
Since $B$ is an isomorphism, so is the dual operator
\[
B'\colon V^\perp(\Omega)\to L^2(\Omega).
\]
Note that $V^\perp(\Omega)$ is identified with $(W/V)'(\Omega).$ The operator $B'$ is usually called de Rham's operator, although ``pressure operator'' would be a more suitable name. Its properties are summarized in

\begin{cor}[De Rham--Tartar]\label{dRm2:thm}
Assume $\Gamma_D$ and $\Gamma_N$ as in Theorem~\ref{main:thm}(iii). Suppose $L$ in $W'(\Omega)$ satisfies
\begin{equation*}\label{LVperp}
\langle L,v\rangle=0\quad \forall v\in V(\Omega).
\end{equation*}
Then there exists a unique $p$ in $L^2(\Omega)$ such that
\begin{equation*}
\langle L,v\rangle=\int_\Omega p\,\div v\,dx\quad \forall v\in W(\Omega).
\end{equation*}
Moreover 
\begin{equation*}
\norm{p}_{L^2(\Omega)}\leq C\norm{L}_{W'(\Omega)}
\end{equation*}
where $C$ is the constant of Theorem~\ref{range_div:thm}.
\end{cor}

\begin{proof}
Clearly $p=B'L.$ Since $\norm{B'}=\norm{B}$ the result follows.
\end{proof}


Next, we show that one can construct a divergence free ``lift'' of any trace in $H^{1/2}(\Gamma_D;\R^n).$

\begin{cor}[Lift operator]\label{lift:cor}
Assume $\Gamma_D$ and $\Gamma_N$ as in Theorem~\ref{main:thm}(iii). For any $h\in H^{1/2}(\Gamma_D;\R^n)$ there exists $u\in H^1(\Omega;\R^n)$ such that
\begin{equation}\label{lift:eq}
\left\{
\begin{aligned}
\div u& =0\quad \text{in }\Omega\\
u& =h\quad \text{on } \Gamma_D.
\end{aligned}\right.
\end{equation}
Moreover $u$ can be chosen so that
\begin{equation}\label{lift:est}
\norm{u}_{H^{1}(\Omega)}\leq C' \norm{h}_{H^{1/2}(\Gamma_D)}
\end{equation}
where the constant $C'$ depends only on $\Omega$ and $\Gamma_D.$
\end{cor}

\begin{proof}
Choose any $v\in H^1(\Omega;\R^n)$ such that $v=h$ on $\Gamma_D.$ According to Theorem~\ref{range_div:thm} there exists $w\in H^1(\Omega;\R^n)$ such that
\begin{equation*}
\left\{
\begin{aligned}
\div w& =\div v\quad \text{in }\Omega\\
w& =0\quad \text{on } \Gamma_D
\end{aligned}\right.
\end{equation*}
and
\[
\norm{w}_{H^1}<2\norm{\{w\}}_{W/V}
\]
Then $u=v-w$ satisfies \eqref{lift:eq}. Moreover
\begin{equation*}
\norm{u}_{H^1}\leq \norm{v}_{H^1}+\norm{w}_{H^1}< \norm{v}_{H^1}+2C\norm{\div v}_{L^2}\leq (1+2C\sqrt{n})\norm{v}_{H^1},
\end{equation*}
where $C$ as in Theorem~\ref{range_div:thm}. Taking the infimum over all $v$ such that $v=h$ on $\Gamma_D$ gives \eqref{lift:est}.
\end{proof}

\section{Proof of main result}\label{proof:sec}

We give here the proof of Theorem~\ref{main:thm}(iii), i.e. for the case that
\[
\int_{\Gamma_D}dS>0\quad \text{and}\quad \int_{\Gamma_N}dS>0.
\]
For the proof of Theorem~3.1(ii) we refer to Chapter IV of \cite{Boy13}. Note that in view of Remark~\ref{psi0:rem}, $\partial \Omega$ need not be of class $C^{1,1}$ as stated in \cite{Boy13}. For simplicity we take $\mu=1/2.$

Using the results proved in Sections~\ref{ineq:sec}--\ref{vector_analysis:sec}, notably Theorem~\ref{Korn3:thm} and Corollary~\ref{dRm2:thm}, the existence and uniqueness of weak solutions of \eqref{Stokes:bvp} follows from standard applications of the Lax--Milgram theorem. It is only for the reader's convenience that we provide all details. To this end, consider the symmetric bilinear form
\begin{equation*}
a(u,v)=\int_\Omega \ee(\nabla u):\ee(\nabla v)\,dx,
\end{equation*}
which satisfies
\begin{equation*}
\begin{aligned}
\abs{a(u,v)}& \leq \norm{u}_{H^1}\norm{v}_{H^1}\\
a(u,u)& \geq K^{-2}\norm{u}_{H^1}^2
\end{aligned}\quad
(u,v\in W(\Omega)),
\end{equation*}
where $K>0$ is the constant in the Korn inequality \eqref{Korn3:ineq}.
According to the Lax--Milgram theorem, for any $L\in W'(\Omega)$ there exists a unique element $u$ in $V(\Omega)$ such that
\begin{equation}\label{LM}
a(u,v)=\langle L,v\rangle \quad \forall v\in V(\Omega).
\end{equation}
Moreover, $u$ satisfies
\begin{equation}\label{u_std:est}
\norm{u}_{H^1}\leq K^2\norm{L}_{V'}.
\end{equation}
Let $A\colon W(\Omega)\to W'(\Omega)$ denote the bounded linear operator defined by
\[
\langle Au,v\rangle=a(u,v)\quad \forall v\in W(\Omega).
\]
Then \eqref{LM} is equivalent to
\[
Au=L\quad \text{in }V'(\Omega).
\]
Let $F$ in $V^\perp(\Omega)$ be defined by
\[
F=Au-L\quad \text{in }W'(\Omega).
\]
Set $p=B'F,$ where $B'$ is the pressure operator defined in Section~\ref{vector_analysis:sec}. Then
\begin{equation*}
\norm{p}_{L^2}\leq \norm{B'}\norm{F}_{W'}=C\norm{Au-L}_{W'},
\end{equation*}
where $C$ is the constant of Theorem~\ref{range_div:thm}. Since $u$ belongs to the null space of $Au-L,$ we have
\[
\norm{Au-L}_{W'}=\sup_{\substack{v\in W(\Omega)\\ \norm{v}=1}} \abs{\langle Au-L,v\rangle}=\sup_{\substack{v\in W(\Omega)\\ \norm{v}=1,\:a(u,v)=0}} \abs{\langle L,v\rangle}\leq \norm{L}_{W'}.
\]
Hence
\begin{equation}\label{p_std:est}
\norm{p}_{L^2}\leq C\norm{L}_{W'}.
\end{equation}

\begin{proof}[Proof of Theorem~\ref{main:thm}(iii)]
Using the linearity of the system (\ref{Stokes:bvp}) we shall construct $u$ and $p$ as superpositions in four steps. 

1.\enspace Assume (temporarily) that $g=0$ and $h=0.$ Then $u\in V(\Omega)$ and \eqref{Stokes:weak} can be written as
\begin{equation}\label{uf:weak}
\int_\Omega -p\,\div v+\ee(\nabla u):\nabla v+f\cdot v\,dx=0\quad \forall v\in W(\Omega).
\end{equation}
Clearly \eqref{uf:weak} follows from the abstract result following \eqref{LM} by taking
\begin{equation*}
\langle L,v\rangle=-\int_\Omega f\cdot v\,dx.
\end{equation*}
Thus
\begin{align}
\norm{u}_{H^1}& \leq K^2\norm{L}_{V'}\leq K^2\norm{f}_{L^2}\label{uf:est}\\
\norm{p}_{L^2}& \leq C\norm{L}_{W'}\leq C\norm{f}_{L^2}\label{pf:est}.
\end{align}

2.\enspace Assume $f=0$ and $h=0.$ Then $u\in V(\Omega)$ and \eqref{Stokes:weak} can be written as
\begin{equation}\label{ug:weak}
\int_\Omega -p\,\div v+\ee(\nabla u):\nabla v\,dx=\int_{\Gamma_N}g\cdot v\,dS\quad \forall v\in V(\Omega).
\end{equation}
As above, \eqref{ug:weak} follows from the abstract result by choosing
\begin{equation*}
\langle L,v\rangle=\int_{\Gamma_N}g\cdot v\,dS,
\end{equation*}
which also defines an element in $H^{-1/2}(\Gamma_N;\R^n).$ Using the representation \eqref{GN:dual} we obtain
\begin{equation}\label{Tg_var:eq}
\int_{\Gamma_N}g\cdot v\,dS
=\int_{\Omega} G:\nabla v+(\div G)\cdot v\,dx\quad \forall v\in W(\Omega),
\end{equation}
for some $G\in L^2(\Omega;\R^{n\times n})$ with $\div G\in L^2(\Omega;\R^n).$ In other words $g=G\hat{n}$ on $\Gamma_N.$ Thus we see that
\[
\norm{L}_{W'(\Omega)}\leq \norm{g}_{H^{-1/2}(\Gamma_N)}
\]
which implies
\begin{align}
\norm{u}_{H^1}& \leq K^2\norm{L}_{V'}\leq K^2\norm{g}_{H^{-1/2}}\label{ug:est}\\
\norm{p}_{L^2}& \leq C\norm{L}_{W'}\leq C\norm{g}_{H^{-1/2}}\label{pg:est}.
\end{align}

3.\enspace Assume $f=0$ and $g=0.$ Then \eqref{Stokes:weak} can be written as
\begin{equation}\label{uh:weak}
\int_\Omega -p\,\div v+\ee(\nabla u):\nabla v\,dx=0\quad \forall v\in W(\Omega).
\end{equation}
Since $u$ is now assumed to satisfy $u=h$ on $\Gamma_D,$ we proceed in a slightly different manner by seeking $u$ in the admissible class
\begin{equation*}
V_h=\left\{v\in H^1(\Omega;\R^n):\div v=0\text{ in }\Omega,\quad v=h\text{ on }\Gamma_D\right\}.
\end{equation*}
By Corollary~\ref{lift:cor},  $V_h$ is not empty. Since $V_h$ is a closed convex set, there exists a unique $u\in V_h$ such that 
\begin{equation}\label{uvp}
\int_\Omega \abs{\ee(\nabla u)}^2 dx\leq \int_\Omega \abs{\ee(\nabla v)}^2 dx\quad \forall v\in V_h.
\end{equation}
It is readily checked that \eqref{uvp} is equivalent to
\[
a(u,v)=0\quad \forall v\in V(\Omega).
\]
This implies \eqref{uh:weak} with $p=B'Au.$ By Korn's inequality and \eqref{uvp} we have
\begin{equation*}
 \norm{u-v}_{H^1}\leq K\norm{\ee(\nabla u)-\ee(\nabla v)}_{L^2}\leq 2K\norm{\ee(\nabla v)}_{L^2}\quad \forall v\in V_h.
\end{equation*}
Thus
\begin{equation*}
\norm{u}_{H^1}\leq \norm{u-v}_{H^1}+\norm{v}_{H^1}\leq (2K+1)\norm{v}_{H^1}\quad \forall v\in V_h.
\end{equation*}
Choosing $v$ as in Corollary~\ref{lift:cor} we deduce
\begin{equation}
\norm{u}_{H^1}\leq C'(2K+1)\norm{h}_{H^{1/2}}.
\end{equation}
The pressure satisfies
\begin{equation}\label{ph:est}
\norm{p}_{L^2}\leq C\norm{Au}_{W'}\leq CC'(2K+1)\norm{h}_{H^{1/2}}.
\end{equation}

4.\enspace Let $(u_f,p_f),$ $(u_g,p_g)$ and $(u_h,p_h)$ denote the solutions of \eqref{uf:weak}, \eqref{ug:weak} and \eqref{uh:weak} respectively. Then 
\begin{equation*}
u=u_f+u_g+u_h,\quad p=p_f+p_g+p_h
\end{equation*}
solve the weak formulation \eqref{Stokes:weak}. Moreover, the above estimates imply 
\begin{equation*}
\norm{u}_{H^1}+\norm{p}_{L^2}\leq C(\norm{f}_{L^2}+\norm{g}_{H^{-1/2}}+\norm{h}_{H^{1/2}}),
\end{equation*}
where the constant $C$ depends only on $\Omega$ and $\Gamma_D$
Hence $u$ and $p$ are uniquely determined.
\end{proof}

\section{Pressure-driven flow between two parallel plates}\label{app:sec}

\begin{figure}[htbp]
\begin{center}
\includegraphics[width=0.6\textwidth]{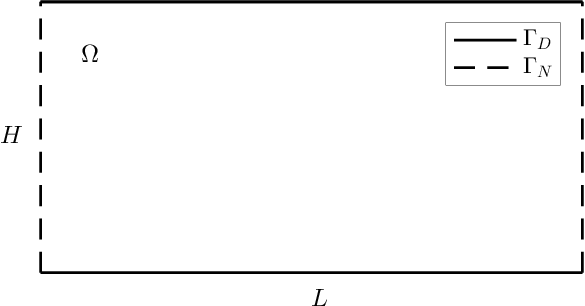}
\caption{Boundary of rectangular fluid domain $\Omega$}
\label{Omega:fig}
\end{center}
\end{figure}

We choose a simple application to illustrate Theorem~\ref{main:thm}(iii), namely the two-dimensional steady flow between two parallel plates driven by a ``pressure gradient''. There seems to be no general consensus as to the meaning of the term ``pressure-driven flow'' in the fluid dynamics literature. By pressure-driven flow, we mean the flow induced by a \emph{normal stress distribution} alone. For alternative meanings we refer to \cite{Con94} and \cite{Hey96}.

Let $\Omega\subset \R^2$ be the rectangular domain $(0,L)\times (0,H),$ i.e.
\begin{equation*}\label{Omega_rect:eq}
\Omega=\bigl\{ (x_1,x_2): 0<x_1<L,\quad 0<x_2<H\bigr\}
\end{equation*}
with
\begin{equation}\label{Gamma_rect:eq}
\Gamma_N=\Gamma_1\cup \Gamma_2,\quad 
\Gamma_D=\Gamma_3\cup \Gamma_4,
\end{equation}
where
\begin{align*}
\Gamma_1& =\bigl\{ (x_1,x_2): x_1=0,\quad 0<x_2<H\bigr\}\\
\Gamma_2& =\bigl\{ (x_1,x_2): x_1=L,\quad 0<x_2<H\bigr\}\\
\Gamma_3& =\bigl\{ (x_1,x_2): x_2=0,\quad 0\leq x_1\leq L\bigr\}\\
\Gamma_4& =\bigl\{ (x_1,x_2): x_2=H,\quad 0\leq x_1\leq L\bigr\}.
\end{align*}
See Figure~\ref{Omega:fig} for an illustration.

Assume $f=0,$ $h=0$ in \eqref{Stokes:bvp} and define
\begin{equation}\label{g_normal_stress:eq}
g(x)=
\begin{cases}
-p_{\mathrm{in}}\,\hat{n} & \text{if }x\in \Gamma_1\\
-p_{\mathrm{out}}\,\hat{n} & \text{if } x\in \Gamma_2,
\end{cases}
\end{equation}
where $p_{\mathrm{in}}$ (inlet pressure) and $p_{\mathrm{out}}$ (outlet presssure) are given constants. Thus, we have a normal stress condition on the lateral boundaries. According to Theorem~\ref{main:thm}(iii) the b.v.p. \eqref{Stokes:bvp} has a unique solution $u\in H^1(\Omega;\R^2),$ $p\in L^2(\Omega)$ such that
\begin{equation*}
\norm{u}_{H^1(\Omega)}+\norm{p}_{L^2(\Omega)}\leq C\norm{g}_{H^{-1/2}(\Gamma_N)}.
\end{equation*}
We compute $u$ and $p$ using the ``Creeping Flow'' module in Comsol Multiphysics, which is a software based on the finite element method, for
\begin{equation*}
L=2,\quad H=1,\quad p_{\mathrm{in}}=1,\quad p_{\mathrm{out}}=0,\quad \mu=1.
\end{equation*}
It is instructive to compare $(u,p)$ with the well-known Poiseuille solution $(\tilde{u},\tilde{p})$ (see e.g. Chapter 7 of \cite{Pan96}) defined by 
\begin{equation}\label{Poiseuille:eq}
\left \{
\begin{aligned}
\tilde{u}(x)& =-\frac{x_2(H-x_2)}{2\mu}\nabla \tilde{p}(x)\\
\tilde{p}(x)& =p_{\mathrm{out}}\,\frac{x_1}{L}+p_{\mathrm{in}}\,\biggl(1-\frac{x_1}{L}\biggr)
\end{aligned}\right.
\quad (x\in \Omega).
\end{equation}
Observe that $(\tilde{u},\tilde{p})$ also satisfies (\ref{Stokes:bvp}a,b,d) but not (\ref{Stokes:bvp}c) as
\begin{equation*}
\tilde{p}=
\begin{cases}
p_{\mathrm{in}} & \text{ on }\Gamma_1\\
p_{\mathrm{out}} & \text{ on }\Gamma_2,
\end{cases}
\end{equation*}
$\tilde{u}_2=0$ and
\begin{align*}
2\mu\,\ee(\nabla \tilde{u})& =\mu
\begin{bmatrix}
2\dfrac{\partial \tilde{u}_1}{\partial x_1} & \dfrac{\partial \tilde{u}_1}{\partial x_2}+\dfrac{\partial \tilde{u}_2}{\partial x_1}\\
\dfrac{\partial \tilde{u}_2}{\partial x_1}+\dfrac{\partial \tilde{u}_1}{\partial x_2} & 2\dfrac{\partial \tilde{u}_2}{\partial x_2}
\end{bmatrix}\\
& =\frac{(H-2x_2)(p_{\mathrm{in}}-p_{\mathrm{out}})}{2L}
\begin{bmatrix}
0 & 1\\
1 & 0
\end{bmatrix}.
\end{align*}
Thus $(u,p)$ and $(\tilde{u},\tilde{p})$ are not identical. Nevertheless the velocity profiles of $u$ and $\tilde{u}$ are similar (compare Figures~\ref{vel1:fig} and \ref{vel2:fig}) as are the boundary stresses (compare Figures~\ref{stress1:fig} and \ref{stress2:fig}). Note that, for the sake of visibility, velocity vectors are scaled by a factor $4.0$ while stress vectors are scaled by a factor $0.25.$ The pressure distribution $p$ behaves like the linear function $\tilde{p}=\tilde{p}(x_1)$ around $x_1=1$ but deviates more and more from $\tilde{p}$ as one approaches the lateral boundaries (see Figures~\ref{p1:fig} and \ref{p2:fig}). Note also that the viscous stresses on $\Gamma_N$ are $x_1$-directional in Figure~\ref{visc_stress1:fig} while they are $x_2$-directional in Figure~\ref{visc_stress2:fig}.

To summarize, our numerical example suggests that the Poiseuille solution $(\tilde{u},\tilde{p})$ defined by \eqref{Poiseuille:eq} \emph{approximates} the normal stress solution $(u,p).$ In fact, since
\begin{equation*}
2\mu\,\ee(\nabla\tilde{u})\hat{n}=O\biggl(\frac{H}{L}\biggr)
\end{equation*}
one may expect them to be asymptotically equivalent as $H/L\to 0,$ i.e. if $\Omega$ is ``infinitely long'' or ``infinitely thin''. A precise statement, when $H\to 0$ and $L$ is constant, follows.

\begin{thm}\label{asymptotic:thm}
Let $\Omega$ be the rectangle $(0,L)\times (0,H)$ with $\Gamma_N$ and $\Gamma_D$ as in \eqref{Gamma_rect:eq}. Let $(u,p)$ be the solution of \eqref{Stokes:bvp} with $f=0,$ $h=0$ and
\begin{equation*}
g(x)=-\tilde{p}(x)\hat{n},\quad \tilde{p}(x)=p_{\mathrm{out}}\,\frac{x_1}{L}+p_{\mathrm{in}}\,\biggl(1-\frac{x_1}{L}\biggr)\quad (x\in \Gamma_N).
\end{equation*}
Then
\begin{align*}
\lim_{H\to 0} \frac{1}{\abs{\Omega}}\int_{\Omega} \frac{L}{H^2}\,u(x)\phi\Bigl(\frac{x_1}{L},\frac{x_2}{H}\Bigr) dx& = \int_{\square} u^0(y)\phi(y)\,dy\\
\lim_{H\to 0} \frac{1}{\abs{\Omega}}\int_{\Omega} p(x)\phi\Bigl(\frac{x_1}{L},\frac{x_2}{H}\Bigr) dx& =\int_{\square} p^0(y)\phi(y)\,dy
\end{align*}
for all $\phi\in C(\R^2),$ where $\square$ is the unit square $(0,1)\times (0,1)$ and
\begin{equation*}
\left\{
\begin{aligned}
u^0(y)& =-\frac{y_2(1-y_2)}{2\mu}\nabla p^0(y)\\
p^0(y)& =p_{\mathrm{out}}\,y_1+p_{\mathrm{in}}\,(1-y_1)
\end{aligned}\right.
\qquad (y\in \square).
\end{equation*}
\end{thm}

We shall not prove Theorem~\ref{asymptotic:thm} here as a more general result, connected with lubrication theory, will appear in a forthcoming paper. The notion of convergence used in Theorem~\ref{asymptotic:thm} is called ``two-scale convergence for thin domains''. It was introduced by Maru\v{s}i\'{c} and Maru\v{s}i\'{c}-Paloka in \cite{MarPal00}.

%

\begin{figure}
\centering
\includegraphics[width=0.6\textwidth]{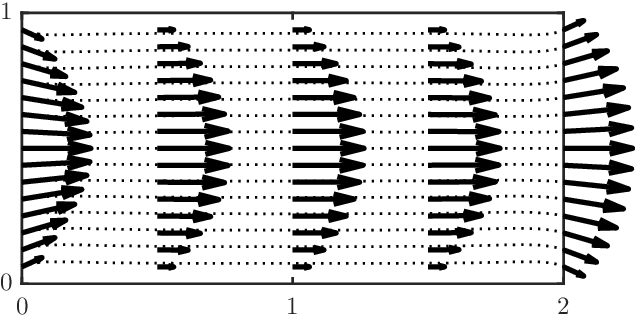}
\caption{Velocity $u$ (dotted streamlines)}
\label{vel1:fig}
\end{figure}

\begin{figure}
\centering
\includegraphics[width=0.6\textwidth]{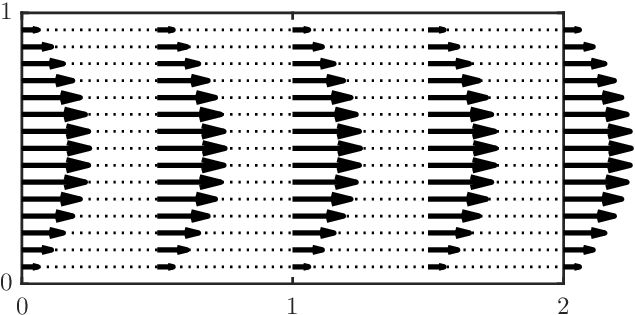} 
\caption{Velocity $\tilde{u}$ (dotted streamlines)}
\label{vel2:fig}
\end{figure}

\begin{figure}
\centering
\includegraphics[width=0.6\textwidth]{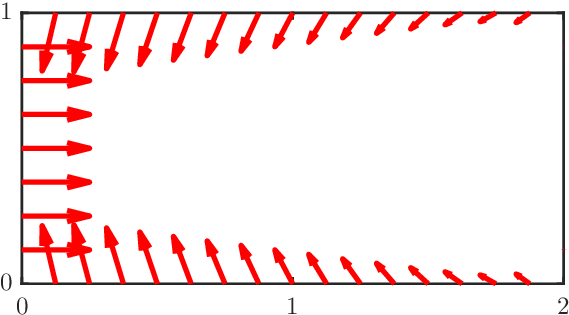} 
\caption{Total stress $\left(-p\, I+2\mu\,\ee(\nabla u)\right)\hat{n}$}
\label{stress1:fig}
\end{figure}

\begin{figure}
\centering
\includegraphics[width=0.6\textwidth]{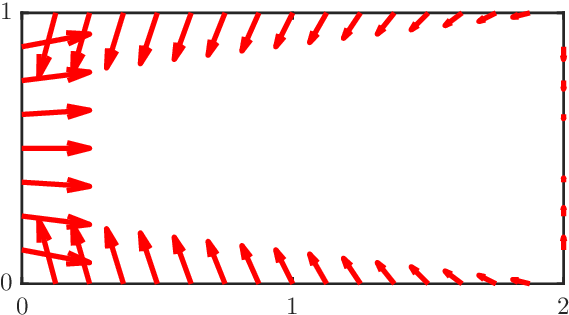} 
\caption{Total stress $\left(-\tilde{p}\, I+2\mu\,\ee(\nabla \tilde{u})\right)\hat{n}$}
\label{stress2:fig}
\end{figure}

\begin{figure}
\centering
\includegraphics[width=0.6\textwidth]{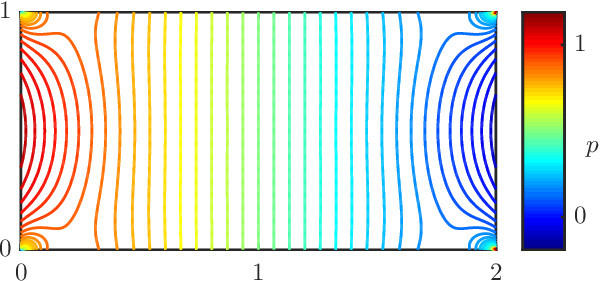} 
\caption{Pressure $p$}
\label{p1:fig}
\end{figure}

\begin{figure}
\centering
\includegraphics[width=0.6\textwidth]{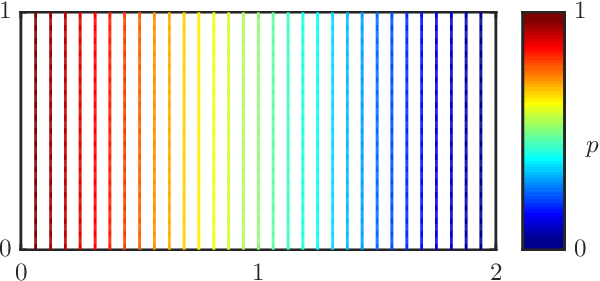} 
\caption{Pressure $\tilde{p}$}
\label{p2:fig}
\end{figure}

\begin{figure}
\centering
\includegraphics[width=0.6\textwidth]{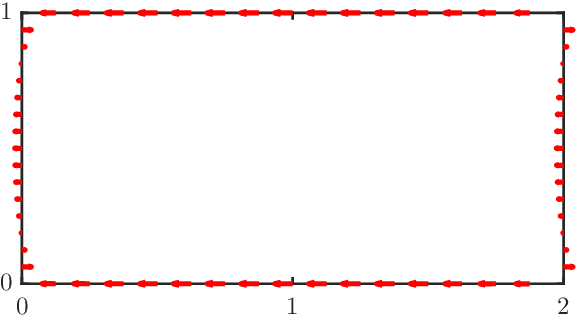}
\caption{Viscous stress $2\mu\,\ee(\nabla u)\hat{n}$}
\label{visc_stress1:fig}
\end{figure}

\begin{figure}
\centering
\includegraphics[width=0.6\textwidth]{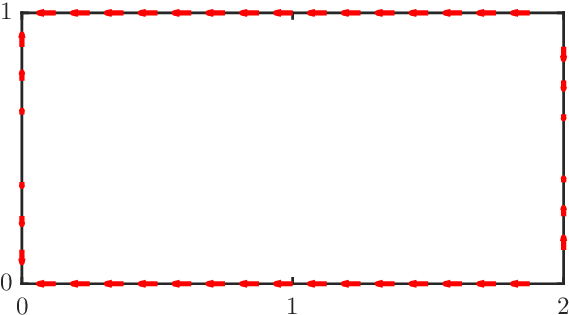}
\caption{Viscous stress $2\mu\,\ee(\nabla \tilde{u})\hat{n}$}
\label{visc_stress2:fig}
\end{figure}

\end{document}